\documentclass[leqno]{amsart}

\usepackage{stmaryrd,graphicx}
\usepackage{amssymb,mathrsfs,amsmath,amscd,amsthm,color}
\usepackage{float}
\usepackage{mathabx}
\usepackage[all,cmtip]{xy}
\DeclareMathAlphabet{\mathpzc}{OT1}{pzc}{m}{it}
\usepackage{amsfonts,latexsym,wasysym}
\usepackage[bookmarks=true, bookmarksnumbered=true, breaklinks=true, pdfstartview=FitH, hyperfigures=false, plainpages=false, naturalnames=true, colorlinks=false,pagebackref=false, pdfpagelabels]{hyperref}
\usepackage{cleveref}
\usepackage{tikz-cd}
\usepackage{enumitem}
\setlist[enumerate]{label={\upshape(\arabic*)}}

\DeclareMathOperator{\diam}{diam}

\DeclareMathOperator{\supp}{supp}
\DeclareMathOperator{\card}{card}
\DeclareMathOperator{\fin}{fin}

\DeclareMathOperator{\MCP}{MCP}
\DeclareMathOperator{\VR}{VR}

\DeclareMathOperator{\C}{\check{C}}
\DeclareMathOperator{\id}{id}
\DeclareMathOperator{\Vect}{\normalfont{\textbf{Vect}}}
\DeclareMathOperator{\Top}{\normalfont{\textbf{Top}}}

\newcommand{\gvr}{|{\VR(X;r)}|}
\newcommand{\vu}{\mathcal{V}(\mathscr{U})}
\newcommand{\vm}{\mathcal{V}^m(\mathscr{U})}
\newcommand{\gv}{|\mathcal{V}(\mathscr{U})|}
\newcommand{\mgv}{|\mathcal{V}(\mathscr{U})|_m}

\newcommand{\p}{\partial}
\newcommand{\wt}{\widetilde}

\newcommand{\mci}{\mathcal{I}}
\newcommand{\mcp}{\mathcal{P}}
\newcommand{\mcs}{\mathcal{S}}

\newcommand{\mcv}{\mathcal{V}}
\newcommand{\scra}{\mathscr{A}}
\newcommand{\scrc}{\mathscr{C}}
\newcommand{\scru}{\mathscr{U}}

\newcommand{\bbk}{\mathbb{K}}
\newcommand{\bbn}{\mathbb{N}}

\newcommand{\bbr}{\mathbb{R}}

\newcommand{\bbz}{\mathbb{Z}}

\newtheorem{introthm}{Theorem}
\newtheorem{introcor}[introthm]{Corollary}
\newtheorem{recallthm}{Theorem}
\newtheorem{recallcor}[recallthm]{Corollary}
\newtheorem{theorem}{Theorem}[section]
\newtheorem{lemma}[theorem]{Lemma}
\newtheorem{proposition}[theorem]{Proposition}

\theoremstyle{definition}
\newtheorem{definition}[theorem]{Definition}

\newtheorem{question}[theorem]{Question}
\newtheorem{remark}[theorem]{Remark}

\begin{document}

\title[Vietoris thickenings and complexes are weakly equivalent]{Vietoris thickenings and complexes are weakly homotopy equivalent}

\author[P. Gillespie]{Patrick Gillespie}
\address{University of Tennessee\\ Department of Mathematics\\
Knoxville, TN 37996, USA}
\email{pgilles5@vols.utk.edu}

\date{\today}

\begin{abstract}
Characterizing the homotopy types of the Vietoris--Rips complexes $\VR(X;r)$ of a metric space $X$ is in general a difficult problem. The Vietoris--Rips metric thickening $\VR^m(X;r)$, a metric space analogue of $\VR(X;r)$, was introduced as a potentially more amenable object of study with several advantageous properties, yet the relationship between its homotopy type and that of $\VR(X;r)$ was not fully understood. We show that for any metric space $X$ and threshold $r>0$, the natural bijection $|{\VR(X;r)}|\to \VR^m(X;r)$ between the (open) Vietoris--Rips complex and Vietoris--Rips metric thickening is a weak homotopy equivalence.
\end{abstract}

\maketitle

\section{Introduction}

Persistent homology is one of the central tools of topological data analysis and in recent years has found a number of applications to a diverse range of fields such as neuroscience \cite{Chung, Bendich}, biology and biochemistry \cite{Cang, Maroulas22}, materials science \cite{Xia, Lee, Spannaus}, and the study of sensor networks \cite{DeSilva07}. In order to apply persistent homology to a finite and discrete data set $X$, one must be able to convert $X$ into a filtered simplicial complex for which the persistent homology groups can then be computed. A common method of doing so is through the Vietoris--Rips construction, which associates a simplicial complex $\VR(X;r)$ to a metric space $X$ for each threshold value $r>0$. The simplices of $\VR(X;r)$ are the finite subsets of $X$ with diameter less than $r$, so that as $r$ varies from $0$ to $\diam(X)$, the simplicial complex $\VR(X;r)$ evolves from a discrete set to a fully connected complex---with the intermediate stages containing information relevant to $X$.

Because of the role of Vietoris--Rips complexes in applications of persistent homology, there has been growing interest in understanding the homotopy type of $\gvr$ (the geometric realization of $\VR(X;r)$) for not just a finite discrete set but instead the limiting case of a Riemannian manifold or Euclidean submanifold $X$ for arbitrary threshold value $r>0$ \cite{Circle, Ellipse, Lim}. However, when the Vietoris--Rips complex is considered for a metric space $X$ which is not discrete, the topology of $\gvr$ becomes cumbersome to work with. For one, $\VR(X;r)$ in this case is not locally finite as a simplicial complex, hence $\gvr$ is not metrizable. Even if one equips $\gvr$ with the metric topology for simplicial complexes, its metric has no natural relationship with the metric of $X$.

To remedy some of these shortcomings, the Vietoris--Rips metric thickening $\VR^m(X;r)$ was introduced by Adamaszek, Adams, and Frick in \cite{Metric}. The underlying set of $\VR^m(X;r)$ is the same as $\gvr$, yet $\VR^m(X;r)$ is equipped with a metric which makes the natural function $X\to \VR^m(X;r)$ an isometric embedding. Similar to Hausmann's theorem \cite{Hausmann} which provides the existence of a homotopy equivalence $\gvr\to X$ for sufficiently small values of $r$ when $X$ is a compact Riemannian manifold, it was shown in \cite{Metric} that for a compact Riemannian manifold $X$ and $r$ sufficiently small, there exists a homotopy equivalence $\VR^m(X;r)\to X$. Moreover, the homotopy equivalence $\VR^m(X;r)\to X$ is naturally defined in terms of Fr\'{e}chet means (which are unique in this case due to the choice of sufficiently small $r$) and has the embedding $X\to \VR^m(X;r)$ as its homotopy inverse, whereas the homotopy equivalence of Hausmann's theorem is highly non-canonical in the sense that it is defined in terms of a total ordering of $X$.

The relationship between $\gvr$ and $\VR^m(X;r)$ was studied in further detail in \cite{Adams, Adams2}. Adams, Frick, and Virk showed in \cite{Adams} that $\gvr$ and $\VR^m(X;r)$ have isomorphic homotopy groups for any separable metric space $X$ and arbitrary $r>0$. More precisely, it was shown that the Vietoris complex $\vu$ and Vietoris metric thickening $\vm$, which generalize $\VR(X;r)$ and $\VR^m(X;r)$ respectively, have isomorphic homotopy groups whenever $\scru$ is a uniformly bounded open cover of a separable metric space $X$. Their proof relies on comparing the nerve of a particular cover of $\vm$ with the nerve of $\scru$, the latter of which is homotopy equivalent to $\gv$ by Dowker duality \cite{Dowker}. Critically, this proof does not produce a map between $\gv$ and $\vm$ which induces the isomorphisms of homotopy groups. 

Ultimately, one would like to know when $\gv$ and $\vm$ have the same homotopy type, as this could allow the homotopy type of $\gv$ to be studied through metric techniques applied to $\vm$. As a step in this direction, our main result is to address Question 7.2 of \cite{Adams} and show that the natural bijection $\gv\to \vm$ is a weak homotopy equivalence.

\begin{introthm}\label{mainthm}
For any uniformly bounded open cover $\scru$ of a metric space $X$, the natural bijection $\id:\gv\to \mcv^m(\scru)$ is a weak homotopy equivalence.
\end{introthm}

Since the Vietoris--Rips complex $\VR(X;r)$ and \v Cech complex $\C(X;r)$ are examples of Vietoris complexes $\vu$ taken with respect to particular open covers (see \Cref{Vietoris-def}), \Cref{mainthm} implies that the natural bijections $|{\VR(X;r)}|\to \VR^m(X;r)$ and $|{\C(X;r)}|\to {\C}{}^m(X;r)$ are weak homotopy equivalences.

\Cref{mainthm} can be seen as a significant step towards a proof that $\gv$ and $\vm$ are homotopy equivalent. Whitehead's theorem states that a weak homotopy equivalence $f:X\to Y$ is a homotopy equivalence if $X$ and $Y$ have the homotopy types of CW complexes. Thus, in light of \Cref{mainthm}, to show that $\gv$ and $\vm$ are homotopy equivalent, it suffices to show that $\vm$ has the homotopy type of a CW complex.

The authors of \cite{Adams} asked whether $\gv$ and $\vm$ have isomorphic homology groups. This was shown to be true in \cite{Gill} through an argument involving the Mayer--Vietoris spectral sequence anticipated by the authors of \cite{Adams}. However this fact is also a direct consequence of \Cref{mainthm}, as a weak homotopy equivalence induces an isomorphism of (co)homology groups.

\begin{introcor}\label{maincor}
For any uniformly bounded open cover $\scru$ of a metric space $X$, the natural bijection $\id:\gv\to \mcv^m(\scru)$ induces isomorphisms of (co)homology groups.
\end{introcor}

The only difference between \Cref{maincor} and the result in \cite{Gill} is that the isomorphisms between the homology groups of $\gv$ and $\vm$ are now induced by $\id:\gv\to \mcv^m(\scru)$. However, this has new implications for the relationship between the persistent homology of $\VR(X;\bullet)$ and $\VR^m(X;\bullet)$, as well as the persistent homology of the \v{C}ech complex $\C(X;\bullet)$ and \v{C}ech metric thickening ${\C}{}^m(X;\bullet)$, which we elaborate on in \Cref{persistence}.

\begin{introcor}\label{persistence_cor}
For any metric space $X$, the persistence modules $H_n\circ|{\VR(X;\bullet)}|$ and $H_n\circ \VR^m(X;\bullet)$ (respectively, $H_n\circ|{\C(X;\bullet)}|$ and $H_n\circ{\C}{}^m(X;\bullet)$) are isomorphic.
\end{introcor}

A similar result was proved through different methods by Adams, Memoli, Moy, and Wang \cite[Corollary 5.10]{Adams2}, who have shown that for a totally bounded metric space $X$, the interleaving distance between the persistence modules $H_n\circ|{\VR(X;\bullet)}|$ and $H_n\circ \VR^m(X;\bullet)$ (respectively, $H_n\circ|{\C(X;\bullet)}|$ and $H_n\circ{\C}{}^m(X;\bullet)$) is zero. While this is enough to conclude that $\VR(X;\bullet)$ and $\VR^m(X;\bullet)$ have identical (undecorated) persistence diagrams, their result does not imply that the persistence modules are isomorphic.

\subsection*{Outline}

The structure of this paper is as follows. In \Cref{notation} we establish notation and recall preliminary results concerning Vietoris metric thickenings. In \Cref{topologies}, we discuss $\mgv$, the Vietoris complex with the metric topology for simplicial complexes and its relationship with $\vm$. In \Cref{deforming}, we prove a technical lemma about the existence of homotopies for certain maps of simplices $\Delta^n\to\vm$. In \Cref{triangulation}, we recall the Freudenthal--Kuhn triangulation of $\bbr^n$, which we use to construct a class of triangulations of $I^n$. In \Cref{equivalence} we prove \Cref{mainthm} by showing that if $f:I^n\to \vm$ is a map such that the restriction $f|_{\p I^n}$ is continuous as a map $f|_{\p I^n}:\p I^n\to \mgv$, then $f$ can be homotoped rel.\ $\p I^n$ to a continuous map $g:I^n\to \mgv$. The main step for constructing such a homotopy involves triangulating $I^n$ using the Freudenthal-Kuhn triangulation, which allows us to then iteratively apply the homotopy of \Cref{deforming}. Finally, in \Cref{persistence}, we describe how \Cref{mainthm} implies results about the persistent homology of $\VR^m(X;\bullet)$ and ${\C}{}^m(X;\bullet)$.

\section{Notation and preliminaries}\label{notation}

We use $I$ to denote the unit interval, and $I^n$ to denote the unit $n$-cube $I^n=[0,1]^n$. Let $\Delta^n$ denote the standard $n$-simplex. We write $\p I^n$ and $\p \Delta^n$ for the boundaries of $I^n$ and $\Delta^n$ respectively. Given $A\subseteq X$, we say that $H:X\times I\to Y$ is a homotopy rel.\ $A$ if $H(a,t)=H(a,0)$ for all $t\in I$ and $a\in A$. If $H, G:X\times I\to Y$ are homotopies such that $H(x,1)=G(x,0)$ for all $x\in X$, the concatenation $H\cdot G:X\times I\to Y$ is the homotopy defined by 
$$(H\cdot G)(x,t)=
\begin{cases}
H(x, 2t) & t\in [0, 1/2]\\
G(x, 2t-1) & t\in [1/2, 1].
\end{cases}$$
If $\scru$ is a cover of $X$ and $f:Y\to X$ is a map, let $f^{-1}\scru$ denote the cover of $Y$ given by $f^{-1}\scru=\{f^{-1}(U):U\in \scru\}$. If $K$ is an abstract simplicial complex, let $K_n$ denote the set of $n$-simplices of $K$. We use $K^{(n)}$ to denote the $n$-skeleton of $K$, that is, the set of all simplices of dimension not greater than $n$. The geometric realization $|K|$ is the set of all functions $\alpha:K_0\to I$ which satisfy
\begin{enumerate}
\item $\{v\in K_0:\alpha(v)\neq 0\}$ is a simplex of $K$,
\item $\sum_{v\in K_0}\alpha(v)=1$.
\end{enumerate}
If $\sigma\in K$ is a simplex of $K$, let $|\sigma|\subseteq |K|$ denote the set of all $\alpha\in |K|$ with support in $\sigma$. The topology of $|K|$ is the CW topology, in which a set $U$ is open in $|K|$ if $U\cap |\sigma|$ is open in $|\sigma|\cong \Delta^n$ for every $\sigma\in K$. If $\sigma\in K$ is an $n$-simplex of $K$, let $\p\sigma$ denote the set of all $(n-1)$-simplices contained in $\sigma$. A triangulation $T$ of a space $X$ is a homeomorphism $T:|K|\to X$ for some simplicial complex $K$.

\subsection{Complexes with the metric topology}

For a vertex $v\in K$, the \textit{barycentric coordinate} $\psi_v:|K|\to I$ is the function defined by $\psi_v(\alpha)=\alpha(v)$. For each vertex $v$, $\psi_v$ is continuous \cite[Appendix 1, Corollary 2]{MS82}. The barycentric coordinates can be used to define a metric $d_m$ on the set $|K|$ by setting $d_m(\alpha,\beta)=\sum_{v\in K_0}|\psi_v(\alpha)-\psi_v(\beta)|$. We denote by $|K|_m$ the space whose underlying set is $|K|$ and has the metric topology inherited from the metric $d_m$. The space $|K|$ has a finer topology than $|K|_m$, hence the identity map $|K|\to |K|_m$ is continuous, and is a homeomorphism if and only if $K$ is locally finite. Though $\id:|K|\to |K|_m$ is not a homeomorphism if $K$ is not locally finite, it is always a homotopy equivalence \cite[Appendix 1, Theorem 10]{MS82}. The following is a useful characterization of continuous maps into $|K|_m$.

\begin{lemma}\cite[Appendix 1, Theorem 8]{MS82}\label{coord}
If $K$ is a simplicial complex and $Z$ is a topological space, a function $f:Z\to |K|_m$ is continuous if and only if $\psi_v\circ f:Z\to I$ is continuous for all vertices $v\in K_0$.
\end{lemma}

\subsection{Vietoris complexes}\label{Vietoris-def}

If $(X, d_X)$ is a metric space and $r\in \bbr$, the (open) Vietoris--Rips complex $\VR(X;r)$ is the simplicial complex whose simplices are the finite subsets of $X$ with diameter less than $r$. A similar construction is the \v Cech complex $\C(X;r)$ whose simplices are finite subsets $\{x_1, \dots, x_n\}\subseteq X$ such that there exists $z\in X$ satisfying $d_X(z, x_i)<r$ for all $i\leq n$. Note that for $r\leq 0$, $\VR(X;r)$ and $\C(X;r)$ are both empty sets.

Both of these constructions are generalized by the Vietoris complex. Given a cover $\scru$ of $X$, the \textit{Vietoris complex} $\mcv(\scru)$ is the simplicial complex whose vertex set is $X$ and contains a simplex $\sigma=\{x_1, x_2, \dots, x_n\}\subseteq X$ if $\sigma\subseteq U$ for some $U\in \scru$. If $\scru$ is the cover of $X$ by open sets of diameter less than $r$, then $\vu=\VR(X;r)$. Alternatively, if $\scru$ is the open cover of $X$ by open balls of radius $r$, then $\vu=\C(X;r)$.

\begin{remark}
The open Vietoris--Rips complex is sometimes denoted $\VR_<(X;r)$ in order to distinguish it from the closed Vietoris--Rips complex $\VR_\leq(X;r)$: the complex whose simplices are the finite subsets of $X$ with diameter less than or equal to $r$. While work has been done on understanding the homotopy type of Vietoris--Rips complexes $|{\VR_\leq(X;r)}|$ with the $\leq$ convention and their corresponding metric thickenings $\VR^m_\leq(X;r)$ \cite{Metric, Moy}, the results of the present work only apply to open Vietoris--Rips complexes and their corresponding metric thickenings. Hence the notation $\VR(X;r)$ in the present work will always refer to the open Vietoris--Rips complex.
\end{remark}

A cover $\scru$ of $X$ is said to be \textit{uniformly bounded} if there exists $D<\infty$ such that $\diam(U)\leq D$ for all $U\in\scru$. The covers defining $\VR(X;r)$ and $\C(X;r)$ for any $0<r<\infty$ are examples of uniformly bounded covers.

\subsection{Vietoris metric thickenings}

If $(X,d_X)$ is a metric space, let $\mcp^{\fin}(X)$ denote the set of all probability measures on $X$ with finite support. If $\mu,\nu\in \mcp^{\fin}(X)$, a \textit{coupling between $\mu$ and $\nu$} is a probability measure $\gamma$ on $X\times X$ whose marginals on the first and second factors of $X\times X$ are $\mu$ and $\nu$ respectively, that is, $\gamma$ satisfies $\gamma(A\times X)=\mu(A)$ and $\gamma(X\times A))=\nu(A)$ for all Borel subsets $A\subseteq X$. If $\scrc$ denotes the set of couplings between $\mu,\nu\in \mcp^{\fin}(X)$, the 1-Wasserstein distance between $\mu$ and $\nu$ is
$$d_W(\mu,\nu)=\inf_{\gamma\in \scrc} \int_{X\times X}d_X(x,y)\gamma(dx\times dy).$$

Since we will only consider the 1-Wasserstein distance between finitely supported probability measures, we can alternatively define $d_W(\mu,\nu)$ as an infimum of finite sums in the following manner. For $x\in X$, let $\delta_x$ denote the Dirac probability measure at $x$. Each measure $\mu\in \mcp^{\fin}(X)$ has a unique representation $\mu=\sum_{i\in\mci}a_i\delta_{x_i}$ in which $\supp(\mu)=\{x_i\}_{i\in\mci}\subseteq X$ is a finite set indexed by $\mci$, $a_i>0$ for all $i\in \mci$, and $\sum_{i\in\mci}a_i=1$. If it will not cause confusion, we may omit the finite index set $\mci$ and simply write $\mu=\sum_i a_i\delta_{x_i}$. Now if $\mu=\sum_i a_i\delta_{x_i}$ and $\nu=\sum_j b_j\delta_{y_j}$, a coupling between $\mu$ and $\nu$ can be represented as a sum $\gamma=\sum_{i,j}\gamma_{i,j}\delta_{(x_i, y_j)}$ such that $\gamma_{i,j}\geq 0$ for all $i$ and $j$, $\sum_i \gamma_{i,j}=b_j$ for all $j$, and $\sum_j \gamma_{i,j}=a_i$ for all $i$. With this in mind, the $1$-Wasserstein distance is then
$$d_W(\mu,\nu)=\inf_{\gamma\in \scrc}\sum_{i,j}\gamma_{i,j}d_X(x_i, y_j).$$

The Wasserstein metric is sometimes known as the optimal transport metric. A coupling $\gamma$ between $\mu=\sum_i a_i\delta_{x_i}$ and $\nu=\sum_j b_j\delta_{y_j}$ can be regarded as a transport plan between $\mu$ and $\nu$, with the constants $\gamma_{i,j}$ specifying the amount of mass moved from $x_i$ to $y_j$, and $\gamma_{i,j}d_X(x_i, y_i)$ is the associated cost.

If $\scru$ is a cover of $X$, the \textit{Vietoris metric thickening} $\vm$ is the metric space of all $\mu\in \mcp^{\fin}(X)$ such that $\supp(\mu)$ is contained in an element of $\scru$, equipped with the 1-Wasserstein metric. Note that the map $X\to\vm$ defined by $x\mapsto \delta_x$ is an isometric embedding.

The underlying set of $\vm$ is the same as $\gv$ since any $\alpha\in \gv$ is a function $\alpha:X\to I$ with finite support contained in an element of $\scru$ satisfying $\sum_{x\in X}\alpha(x)=1$, hence can be equally regarded as a finitely supported probability measure $\mu\in\vm$. However, the topology of $\vm$ is coarser than that of $\gv$, so that while the natural bijection $\id:\gv\to\vm$ is continuous \cite[Proposition 6.1]{Metric}, it is not in general a homeomorphism. 

Recalling the barycentric coordinate functions $\psi_v$,  we may on occasion find it useful to represent $\mu\in\vm$ as the sum $\mu=\sum_{x\in X}\psi_x(\mu)\delta_x$, which, while indexed by $X$, has only finitely many non-zero terms.

\begin{lemma}\cite[Lemma 2.1]{Adams}\label{linear}
Let $X$ and $Z$ be metric spaces, and let $f,g:Z\to \mcp^{\fin}(X)$ be continuous maps. Then the linear homotopy $H:Z\times I\to \mcp^{\fin}(X)$ given by $H(z,t)=(1-t)f(z)+tg(z)$ is continuous.
\end{lemma}

In particular, \Cref{linear} implies that if $f,g:Z\to \vm$ are continuous, then whenever the linear homotopy $H:Z\times I\to \vm$ between $f$ and $g$ is well-defined, it is also continuous.

\subsection{Open covers of Vietoris metric thickenings}

Following the terminology of \cite{Adams}, if $U$ is a subset of $X$ and $0<p<1$, we say that a subset $A\subseteq \vm$ has the \textit{mass concentration property} for the pair $(p,U)$ if $\mu(U)>p$ for all $\mu\in A$. In this case, we write that $A$ has $\MCP(p,U)$. If $U\subseteq X$, a subset $A\subseteq \mcp^{\fin}(X)$ is \textit{$U$-pumping convex} if for any $\mu\in A$ and $\nu\in\mcp^{\fin}(X)$ satisfying $\supp(\nu)\subseteq\supp(\mu)\cap U$ we have that $(1-t)\mu+t\nu\in A$ for all $t\in I$.

If $U$ is a subset of $X$, let $M_U\subseteq \mcp^{\fin}(X)$ be the set of all $\mu\in\mcp^{\fin}(X)$ such that $\supp(\mu)\subseteq U$. If $\scru$ is an open cover of $X$, then $M_\scru=\{M_U:U\in\scru\}$ is a cover of $\vm$. However, $M_\scru$ is not necessarily an open cover. In \cite{Adams} it was shown that the cover $M_\scru$ can be ``thickened'' with respect to a choice of $0<p<1$ to construct an open cover $\wt M_\scru$ of $\vm$ with several key properties.

\begin{proposition}\label{cover}
Let $\scru$ be an open cover of $X$. For every $0<p<1$, there exists an open cover $\wt M_\scru=\{\wt M_U\}_{U\in\scru}$ of $\mcv^m(\scru)$ such that for every $U\in\scru$,
\begin{enumerate}
\item $M_U\subseteq \wt M_U$,
\item $\wt M_U$ is $U$-pumping convex,
\item $\wt M_U$ has $\MCP(p, U)$.
\end{enumerate}
\end{proposition}

\begin{proof}
See \cite[pg. 8]{Adams}.
\end{proof}

\begin{remark}
The open covers in \cite{Adams} are defined in such a way that the covers' elements are open and satisfy properties (1)--(3) by construction. However, a simpler way to verify \Cref{cover} is the following. For any $0<p<1$, let $\wt M_\scru$ be the collection of sets $\wt M_U=\{\mu\in \vm: \mu(U)>p\}$. Then it is clear that conditions (1)--(3) hold. Lastly, the sets $\wt M_U$ are open in $\vm$ due to the fact that for any $U\in\scru$ and $\mu\in \wt M_U$, we have $d_X(\supp(\mu)\cap U, U^C)>0$ since $\supp(\mu)$ is a finite set. Hence if $\nu\notin \wt M_U$, then any coupling $\gamma$ between $\mu$ and $\nu$ must transport a mass of at least $\mu(U)-p$ a distance of at least $d_X(\supp(\mu)\cap U, U^C)>0$, or more precisely, $\gamma$ must satisfy
$$\gamma\Big((\supp(\mu)\cap U)\times U^C\Big)\geq \mu(U)-p.$$
From this it follows that $d_W(\mu,\nu)>(\mu(U)-p)d_X(\supp(\mu)\cap U, U^C)$, and since this bound does not depend on $\nu$, it follows that $\mu$ is an interior point of $\wt M_U$. Hence the sets $\wt M_U$ are open in $\vm$.

\end{remark}

We will also need the following lemma, which describes how under certain conditions a subset $A$ with $\MCP(p,U)$ can be mapped continuously into $M_U$.

\begin{lemma}\cite[Lemma 4.3]{Adams}\label{pump}
Let $X$ be a metric space, and let the open set $U\subseteq X$ have finite diameter. Suppose $A\subseteq  \mcp^{\fin}(X)$ has $\MCP(p,U)$ for some $0<p<1$, and let $\phi:X\to [0,1]$ be an $L$-Lipschitz map with $\phi^{-1}(0)=U^C$. Then $f:A\to M_U\subseteq \mcp^{\fin}(X)$ defined as
$$f\Big(\sum_{i\in\mci} a_i\delta_{x_i}\Big)=\sum_{i\in\mci}\Big(\frac{a_i\phi(x_i)}{\sum_{j\in\mci} a_j\phi(x_j)}\Big)\delta_{x_i}$$
is continuous. If $A$ is furthermore $U$-pumping convex, then the homotopy $H:A\times [0,1]\to A$ defined by $H(\mu,t)=(1-t)\mu+tf(\mu)$ is well-defined and continuous.
\end{lemma}

\section{The Vietoris Complex with the Metric Topology}\label{topologies}

Let $\scru$ be an open cover of a metric space $X$. Though we can equip $\gv$ with the metric topology for simplicial complexes to obtain the space $\mgv$, note that this space is distinct from $\vm$. For one, the vertex set $V\subseteq \mgv$ is discrete, whereas the same set with the subspace topology from $\vm$ is homeomorphic to $X$.

Our reason for considering $\mgv$ is the following. Because $\id:|K|\to|K|_m$ is a homotopy equivalence for any simplicial complex $K$, we will prove \Cref{mainthm} by showing that $\id:\mgv\to\vm$ is a weak homotopy equivalence. We make this replacement because the continuity of maps $f:Z\to \mgv$ affords a convenient characterization via \Cref{coord}. 

First, we show that the identity $\mgv\to\vm$ is continuous so long as $\scru$ is uniformly bounded.

\begin{proposition}\label{incl}
If $\scru$ is a uniformly bounded open cover of $X$, then the identity map $\mgv\to\vm$ is continuous.
\end{proposition}

\begin{proof}
Since $\scru$ is uniformly bounded, let $D=\sup_{U\in\scru}\diam(U)$. Let $\mu\in\vm$ and let $0<r<2D$. Let $B= B_{\vm}(\mu, r)$ be the open ball in $\vm$ centered at $\mu$ with radius $r$, and let $B'=B_{\mgv}(\mu, r/D)$ be the open ball in $\mgv$ centered at $\mu$ of radius $r/D$. We claim that $B'\subseteq B$, which will show that the identity $\mgv\to\vm$ is continuous. So let $\nu\in B'$. Then
\begin{equation}\label{eq1}
\sum_{x\in X}|\psi_x(\mu)-\psi_x(\nu)|=\sum_{x\in X}|\mu(x)-\nu(x)|< \frac{r}{D}.
\end{equation}
Note that if $\supp(\mu)$ and $\supp(\nu)$ were disjoint, we would have $\sum_{x\in X}|\mu(x)-\nu(x)|=2$. Since $r/D<2$, it follows that $\supp(\mu)\cap\supp(\nu)\neq\emptyset$.

Let $\gamma$ be any coupling between $\mu$ and $\nu$ which satisfies 
$$\gamma(x,x)= \min\big(\mu(x), \nu(x)\big)$$ 
for all $x \in X$. In other words, $\gamma$ is a transport plan between $\mu$ and $\nu$ which keeps fixed any mass they already have in common. Note that (\ref{eq1}) implies that $\gamma$ has at most $r/(2D)$ of its mass off the diagonal of $X\times X$. Hence
$$\int_{X\times X}d_X(x, y)\gamma(dx\times dy)<\diam\Big(\supp(\mu)\cup\supp(\nu)\Big)\cdot \frac{r}{2D}.$$
Finally, we note that the diameters of $\supp(\mu)$ and $\supp(\nu)$ are each less than $D$, and since $\supp(\mu)\cap\supp(\nu)\neq \emptyset$, it follows that $\diam(\supp(\mu)\cup\supp(\nu))<2D$. Thus
$$d_W(\mu,\nu)\leq \int_{X\times X}d_X(x, y)\gamma(dx\times dy)< r,$$
which shows that $B'\subseteq B$.
\end{proof}

We now show that the linear homotopy of \Cref{linear} and the function defined in \Cref{pump} are continuous with respect to the topology of $\mgv$ under similar conditions.

\begin{lemma}\label{CW-linear}
Let $X$ and $Z$ be metric spaces, let $\scru$ be an open cover of $X$, let $f,g:Z\to \mcv^m(\scru)$ be functions, and suppose the linear homotopy $H:Z\times I\to \mcv^m(\scru)$ given by $H(z,t)=(1-t)f(z)+tg(z)$ is well-defined. If both $f$ and $g$ are continuous with respect to the topology of $\mgv$, then so is $H$.
\end{lemma}

\begin{proof}
If both $f$ and $g$ are continuous with respect to the topology of $\mgv$, then $\psi_v\circ f$ and $\psi_v\circ g$ are continuous for every vertex $v\in X$ of $\vu$. Note that
\begin{align*}
(\psi_v\circ H)(z,t) &= \psi_v\Big((1-t)f(z)+tg(z)\Big)\\
&=(1-t)(\psi_v\circ f)(z)+t(\psi_v\circ g)(z).
\end{align*}
Hence $\psi_v\circ H$ is continuous for each vertex $v$, in which case $H$ is continuous with respect to the topology of $\mgv$ by \Cref{coord}.
\end{proof}

\begin{lemma}\label{CW-cont}
Suppose $A\subseteq\vm$ has $\MCP(p, U)$ for some $p>0$ and open set $U\subseteq X$. Let $\phi:X\to I$ be a continuous map satisfying $\phi^{-1}(0)=U^C$. The map $\Phi:A\to M_U$ defined by
$$\Phi\Big(\sum_{i\in\mci} a_i\delta_{x_i}\Big)=\sum_{i\in\mci}\Big(\frac{a_i\phi(x_i)}{\sum_{j\in\mci} a_j\phi(x_j)}\Big)\delta_{x_i}.$$
is continuous with respect to the topologies of $A$ and $M_U$ regarded as subspaces of $\mgv$.
\end{lemma}

\begin{proof}
Let $A$ and $M_U$ have the subspace topology with respect to $\mgv$. Note that a point $\mu=\sum_ia_i\delta_{x_i}\in\mcv^m(\scru)$ can alternatively be written $\mu=\sum_{x\in X}\psi_{x}(\mu)\delta_{x}$. Then if $v\in \vu$ is a vertex, and  $\mu=\sum_{x\in X}\psi_{x}(\mu)\delta_{x}\in A$, we have
\begin{align}\label{eq2}
\begin{split}
\psi_v\Big(\Phi\Big(\sum_{x\in X}\psi_{x}(\mu)\delta_{x}\Big)\Big) &=\psi_v\Big(\sum_{x\in X}\Big(\frac{\psi_{x}(\mu)\phi(x)}{\sum_{y\in X}\psi_{y}(\mu)\phi(y)}\Big)\delta_{x}\Big)\\ &=\frac{\psi_v(\mu)\phi(v)}{\sum_{y\in X}\psi_{y}(\mu)\phi(y)}.
\end{split}
\end{align}
The fact that $A$ has $\MCP(p,U)$ implies that there exists $y\in U$ such that $\psi_y(\mu)>0$. Then since $\phi(y)>0$ for all $y\in U$, we have that
$$\sum_{y\in X}\psi_{y}(\mu)\phi(y)>0$$
for all $\mu\in A$. Hence $\psi_v\circ\Phi$ is well-defined. Additionally, the function $A\to I$ given by $\mu\mapsto \sum_{y\in X}\psi_{y}(\mu)\phi(y)$ is continuous due to the fact that
$$\Big|\sum_{y\in X}\psi_{y}(\mu)\phi(y)- \sum_{y\in X}\psi_{y}(\nu)\phi(y)\Big| \leq\sum_{y\in X}\phi(y)\big|\psi_y(\mu)-\psi_y(\nu)\big|\leq d_m(\mu, \nu),$$
where $d_m$ is the metric on $\mgv$. Since $\psi_v:A\to I$ is continuous, (\ref{eq2}) implies that $\psi_v\circ\Phi$ is continuous. Since $v$ was arbitrary, \Cref{coord} implies that $\Phi:A\to M_U$ is continuous with respect to the topologies of $A$ and $M_U$ as subspaces of $\mgv$.
\end{proof}

\section{Deforming maps of simplices into Vietoris metric thickenings}\label{deforming}

To show that $\mgv\to\vm$ is a weak homotopy equivalence, our goal is to show that given a map $f:I^n\to \vm$ such that $f|_{\p I^n}$ is continuous as a map $f|_{\p I^n}:\p I^n\to \mgv$, there exists a homotopy $H:I^n\times I\to \vm$ rel.\ $\p I^n$ between $f$ and a continuous map $g:I^n\to \mgv$. To do this, we will construct a homotopy inductively on the $k$-skeleton of a triangulation of $I^n$. In this section, we establish a lemma which will serve as an important tool in constructing this homotopy.

First, we start with a lemma concerning the mass concentration property for compact sets.

\begin{lemma}\label{inner}
Suppose $A\subseteq \mcp^{\fin}_X$ is compact and has $\MCP(p, U)$ for some $0<p<1$ and open set $U\subseteq X$. Then there exists an open set $V\subseteq U$ such that $d(V, U^C)>0$ and $A$ has $\MCP(p, V)$.
\end{lemma}

\begin{proof}
For each $i\in\bbn$, let $V_i=\{x\in U:d(x, U^C)>1/i\}$ and let $A_i=\{\mu\in A:\mu(V_i)>p\}$. Note that for each $i\in\bbn$, $V_i$ is open in $X$. We first show that $A_i$ is open in $A$ for all $i\in\bbn$. Let $i\in\bbn$, let $\mu\in A_i$, and choose $r>0$ satisfying
$$0<r<(\mu(V_i)-p)\cdot d\Big(\supp(\mu)\cap V_i, V_i^C\Big).$$
This is possible since $\supp(\mu)\cap V_i$ is a finite set disjoint from the closed set $(V_i)^C$, hence $d(\supp(\mu)\cap V_i, V_i^C)>0$. If $\nu\in A\setminus A_i$, then $\nu(V_i)\leq p$. In this case, we have that $d_W(\mu,\nu)> r$ as any coupling $\gamma$ between $\mu$ and $\nu$ must move a mass of at least $\mu(V_i)-p$ a distance of at least $d(\supp(\mu)\cap V_i, V_i^C)$, or more precisely, $\gamma$ must satisfy 
$$\gamma\Big((\supp(\mu)\cap V_i)\times V_i^C\Big)\geq \mu(V_i)-p.$$
Hence $B(\mu; r)\cap A\subseteq A_i$, showing that $A_i$ is open in $A$.

Next we show that $\{A_i\}_{i\in\bbn}$ is an open cover of $A$. So let $\mu\in A$. Because $d(\supp(\mu)\cap U, U^C)>0$, if $i\in\bbn$ is chosen large enough so that $1/i<d(\supp(\mu)\cap U, U^C)$, we have that $\supp(\mu)\cap U\subseteq V_i$ and thus $\mu(V_i)=\mu(U)>p$. Thus $\{A_i\}_{i\in\bbn}$ is an open cover of $A$. Since $A$ is compact, we may find a finite subcover $\{A_{i_1}, A_{i_2}, \dots, A_{i_N}\}$. We may assume that $i_1<i_2<\cdots<i_N$, in which case $A_{i_1}\subseteq A_{i_2}\subseteq\cdots\subseteq A_{i_N}$, and thus $A=A_{i_N}$. This implies that for all $\mu\in A$, $\mu(V_{i_N})>p$, that is, $A$ has $\MCP(p, V_{i_N})$. By construction, $V_{i_N}$ satisfies $d(V_{i_N}, U^C)>0$.
\end{proof}

We are now ready to state and prove the main lemma of this section. Its statement and proof are inspired by \cite[Prop. 4.4]{Adams}, though with alterations specific to our approach.

\begin{lemma}\label{deform}
Let $\scru$ be a uniformly bounded open cover of a metric space $X$. Let $N\in\bbn$ with $N\geq 2$, and  let $\wt M_\scru$ be the open cover of $\mcv^m(\scru)$ constructed with respect to a choice of $p$ satisfying $1-1/N<p<1$. Let $U_1, \dots, U_n$ be a collection of elements of $\scru$ with $n\leq N$, and set $U=\ \bigcap_{i\leq n}U_i$. Suppose $V\subseteq U$ is an open set which satisfies $d(V, U^C)>0$.  Then given a map
$$g:(\Delta^k,\p\Delta^k)\to (\bigcap_{i\leq n}\wt M_{U_i},M_V),$$
there exists a homotopy $H:\Delta^k\times I\to \bigcap_{i\leq n}\wt M_{U_i}$ rel.\ $\p\Delta^k$ between $g$ and a map $g':\Delta^k\to M_{V'}$, where $V'\subseteq U$ is an open set which satisfies $d(V', U^C)>0$. Moreover, if $g$ is continuous with respect to the topology of $\mgv$, then we may assume $H$ is as well.
\end{lemma}

\begin{proof}
Since $1-1/N<p<1$, we have that $0<N(1-p)<1$. Since each $\wt M_{U_i}$ has $\MCP(p,U_i)$, if $\mu\in\wt M_{U_i}$, then $\mu(U_i^C)< 1-p$. Now if $\mu\in \bigcap_{i\leq n}\wt M_{U_i}$, it follows that
$$\mu(\bigcup_{i\leq n}U_i^C)< n(1-p)\leq N(1-p)<1.$$
This implies that $\mu(\bigcap_{i\leq n}U_i)>1-N(1-p)>0$, showing that $\bigcap_{i\leq n}\wt M_{U_i}$ has $\MCP(q, U)$ for $q=1-N(1-p)$. Let $A=g(\Delta^k)\subseteq \bigcap_{i\leq n}\wt M_{U_i}$. Then $A$ is compact and has $\MCP(q, U)$, hence by \Cref{inner}, there exists an open set $U'\subseteq U$ such that $d(U', U^C)>0$ and $A$ has $\MCP(q, U')$. Let $\epsilon=d(V, U^C)/2$ and let $V'=U'\cup \bigcup_{x\in V}B(x;\epsilon)$. Then $A$ has $\MCP(q, V')$, $d(V, (V')^C)>0$, and $d(V, U^C)>0$. Then we may find an $L$-Lipschitz function $\phi:X\to I$ for some $L>0$ such that $V\subseteq\phi^{-1}(1)$ and $(V')^C=\phi^{-1}(0)$. \Cref{pump} implies that the map $\Phi:A\to M_{V'}$ defined by 
$$\Phi\Big(\sum_i a_i\delta_{x_i}\Big)=\sum_i\Big(\frac{a_i\phi(x_i)}{\sum_j a_j\phi(x_j)}\Big)\delta_{x_i}$$
is continuous. Set $g'=\Phi\circ g$. Define $H:\Delta^k\times I\to \bigcap_{i\leq n}\wt M_{U_i}$ by $H(v, t)=(1-t)g(v)+tg'(v)$. Note that since $\wt M_{U_i}$ is $U_i$-pumping convex for every $i\leq n$, $\bigcap_{i\leq n}\wt M_{U_i}$ is $U$-pumping convex. Then $H$ is well-defined since $\bigcap_{i\leq n}\wt M_{U_i}$ is $U$-pumping convex and $\supp(g'(v))\subseteq \supp(g(v))\cap U$. Consequently, $H$ is continuous by \Cref{linear}.

Hence $g'$ has image in $M_{V'}$. Moreover, since $V\subseteq \phi^{-1}(1)$, we have that $\Phi(\mu)=\mu$ for all $\mu\in A\cap M_V$, hence $g'|_{\p\Delta^k}=g|_{\p\Delta^k}$. Then by the definition of $H$, we see that $H(z, t)=g(z)$ for all $z\in \p\Delta^k$ and $t\in I$.

Lastly, if $g$ is continuous with respect to the topology of $\mgv$, then so is $g'$ by \Cref{CW-cont}. Then from \Cref{CW-linear}, we also have that $H$ is continuous with respect to the topology of $\mgv$.
\end{proof}

\section{The Freudenthal--Kuhn triangulation}\label{triangulation}

\begin{definition}[Freudenthal--Kuhn triangulation of $\bbr^n$]
For $i\in\{1, 2, \dots, n\}$, let $e_i$ denote the canonical $i$-th basis vector of $\bbr^n$. Let $\pi\in S_n$ be a permutation of the set $\{1, 2, \dots, n\}$. For $x \in \bbz^n\subseteq \bbr^n$, let $\sigma(x, \pi)$ denote the convex hull of the points $v_0, v_1, \dots, v_n$ which are given by the equations $v_0=x$ and $v_i=v_{i-1} + e_{\pi(i)}$. The collection $\{\sigma(x, \pi): x\in \bbz^n, \pi\in S_n\}$ defines a triangulation of $\bbr^n$ \cite[Lemma 3.2]{Todd}, sometimes known as the Freudenthal--Kuhn triangulation of $\bbr^n$ due to its use by H. Freudenthal \cite{Freudenthal} as well as H. Kuhn \cite{Kuhn}.
\end{definition}

Less formally, the Freudenthal--Kuhn triangulation is constructed by first subdividing $\bbr^n$ into unit $n$-cubes of the form
$$Q_x=[x_1, x_1+1]\times\cdots\times [x_n, x_n+1]$$
for $x=(x_1, \dots, x_n)\in\bbz^n$. Each $n$-cube $Q_x$ is then triangulated by $n!$ many $n$-simplices in a canonical way so that the triangulations of each $Q_x$ fit together to define a triangulation of $\bbr^n$.

\begin{proposition}\label{tri}
Let $n\in\bbn$. There exists $\alpha(n)\in\bbn$ such that for every $\epsilon>0$ there exists a triangulation $T:|K|\to I^n$ which satisfies
\begin{enumerate}
\item $\diam T(|\sigma|)\leq \epsilon$ for all $\sigma\in K$,
\item every vertex $v\in K_0$ is contained in at most $\alpha(n)$ distinct $n$-simplices of $K$.
\end{enumerate}
\end{proposition}

\begin{proof}
Set $\alpha(n)=2^n\cdot n!$. Let $\epsilon>0$ and choose $p\in\bbn$ such that $\sqrt{n}/p<\epsilon$. Let $T:|K|\to \bbr^n$ be the Freudenthal--Kuhn triangulation of $\bbr^n$. Note that each $x\in\bbz^n$ is contained in $2^n$ distinct cubes $Q_x$. Moreover, for each $x\in\bbz^n$, the triangulation $T$ identifies $Q_x$ with a subcomplex of $K$ consisting of $n!$ many $n$-simplices. Since the vertex set of $K$ is identified with $\bbz^n$ under $T$, each vertex of $K$ is contained in at most $\alpha(n)=2^n\cdot n!$ distinct $n$-simplices of $K$. Observe that the diameter of each simplex in the triangulation is bounded by the diameter of a unit $n$-cube, which has diameter $\sqrt{n}$. Let $L\subseteq K$ be the subcomplex given by $|L|=T^{-1}([0,p]^n)$. If we compose the restriction of $T$ to $|L|$ with the linear map $[0,p]^n\to I^n$ which scales each coordinate by a factor of $1/p$, we obtain a triangulation $T':|L|\to I^n$ such that $\diam T'(|\sigma|)\leq \sqrt{n}/p< \epsilon$ for all $\sigma \in L$ and such that each vertex of $L$ is contained in at most $\alpha(n)=2^n\cdot n!$ distinct $n$-simplices of $L$.
\end{proof}

\begin{remark}
Barycentric subdivisions can be used to produce triangulations of $I^n$ which satisfy condition $(1)$ of \Cref{tri} for arbitrarily small $\epsilon$. However, since the upper bound on the number of distinct $n$-simplices which intersect at a point would depend on the number of barycentric subdivisions performed, this approach would not allow one to produce a value $\alpha(n)$ which is independent of $\epsilon$ and satisfies condition $(2)$.
\end{remark}

\section{Vietoris metric thickenings and complexes are weakly homotopy equivalent}\label{equivalence}

In this section, $\scru$ will always denote a uniformly bounded open cover of a metric space $X$, so that $\id:\mgv\to\vm$ is continuous by \Cref{incl}. To show that $\id:\mgv\to\vm$ is a weak homotopy equivalence, we use the following characterization of weak homotopy equivalences.

\begin{lemma}\label{weak-lemma}
A map $f:X\to Y$ between spaces $X$ and $Y$ is a weak homotopy equivalence if and only if $f$ satisfies the condition that for every commuting diagram of the form
$$\begin{tikzcd}
\p I^n \arrow[r, "\alpha"] \arrow[d, "i", hook] & X \arrow[d, "f"] \\
I^n \arrow[r, "\beta"]                          & Y               
\end{tikzcd}$$
there exists a map $\phi:I^n\to X$ such that $\alpha=\phi \circ i$ and $\beta\simeq f\circ\phi$ rel.\ $\p I^n$.
\end{lemma}

The statement, in a slightly more general form, can be found in \cite[Theorem 11.12]{Bredon}. The proof amounts to first noting that $f:X\to Y$ is a weak homotopy equivalence if and only if the inclusion $i:X\to M_f$ is a weak homotopy equivalence, where $M_f$ denotes the mapping cylinder of $f$. Then since the above condition on $f$ is equivalent to the condition that $\pi_n(M_f,X)=0$ for all $n\geq 1$ and $\pi_0(X)\to \pi_0(M_f)$ is onto, the lemma follows.

If $L\subseteq K$ is a subcomplex of the simplicial complex $K$, then $|K|\times \{0\}\cup |L|\times I$ is a retract of $|K|\times I$. The following lemma is a slightly sharper statement of this fact, where the condition placed on the retraction follows directly from standard proofs, for example \cite[Proposition 0.16]{Hatcher}.

\begin{lemma}\label{hep}
Let $K$ be a simplicial complex and let $L\subseteq K$ be a subcomplex. Let $i:|K|\times\{0\}\cup |L|\times I\to |K|\times I$ denote the inclusion. There exists a retraction $r:|K|\times I\to |K|\times\{0\}\cup |L|\times I$ such that $i\circ r(|\sigma|\times I)\subseteq |\sigma| \times I$ for all $\sigma\in K$.
\end{lemma}

\begin{definition}
If $K$ is a simplicial complex and $\scra$ is a cover of $|K|$, we say that $K$ is \textit{subordinate to $\scra$} if for every simplex $\sigma\in K$, $|\sigma|$ is contained in some element of $\scra$. Additionally, given a cover $\scra$ of a space $X$, we say that a triangulation $T:|K|\to X$ is \textit{subordinate to $\scra$} if $T(|\sigma|)$ is contained in some element of $\scra$ for every simplex $\sigma\in K$.
\end{definition}

To prove that the bijection $\mgv\to\vm$ satisfies the condition of \Cref{weak-lemma}, we will show that for any map $f:I^n\to \vm$ such that $f|_{\p I^n}$ is continuous as a map $f|_{\p I^n}:\p I^n\to \mgv$, there exists a (free) homotopy $H:I^n\times I\to \vm$ between $f$ and a map $f': I^n\to \mgv$, and where $H$ satisfies the condition that $H|_{\p I^n\times I}:\p I^n\times I\to \mgv$ is continuous. Then, as we will describe in the proof of \Cref{lifting}, $H$ will induce a homotopy rel. $\p I^n$ between $f$ and a map $f'':I^n\to \mgv$.

The construction of the homotopy $H$ consists of two main steps. For the first step, we homotope $f$ to a map $g:I^n\to \vm$ such that there exists a triangulation of $I^n$ subordinate to $g^{-1}M_\scru$. This would be unnecessary if $M_\scru$ were an open cover of $\vm$, however, we remind the reader that this is not the case in general. Thus we instead start with a triangulation of $I^n$ subordinate to $f^{-1}\wt M_\scru$ (since $\wt M_\scru$ \textit{is} an open cover of $\vm$) and carefully deform $f$ to a map $g$ which carries each simplex of the triangulation into an element of $M_\scru$.

\begin{lemma}\label{injpart1}
Let $n\in\bbn$ and let $f:I^n\to \mcv^m(\scru)$ be a map such that $f|_{\p I^n}$ is continuous with respect to the topology of $\mgv$. There exists a homotopy $H:I^n\times I\to \mcv^m(\scru)$ between $f$ and a map $g$ such that 
\begin{enumerate}
\item $H|_{\p I^n\times I}$ is continuous with respect to the topology of $\mgv$,
\item $I^n$ admits a triangulation subordinate to $g^{-1}M_\scru$.
\end{enumerate}
\end{lemma}

\begin{proof}
Let $\alpha(n)$ be the integer guaranteed to exist by \Cref{tri} and let $1-1/\alpha(n)<p<1$. Let $\wt M_\scru$ be the open cover of $\mcv^m(\scru)$ from \Cref{cover} obtained with respect to the value $p$. By the Lebesgue number lemma, there exists $\epsilon>0$ such that if $A$ is a subset of $I^n$ with $\diam A<\epsilon$, then $f(A)$ is contained in some element of $\wt M_\scru$. By the construction of the integer $\alpha(n)$, there exists a triangulation $T:|K|\to I^n$ such $\diam T(|\sigma|)<\epsilon$ for all $\sigma\in K$, hence $T$ is subordinate to $f^{-1}\wt M_\scru$, and such that each vertex of $K$ is contained in at most $\alpha(n)$ many $n$-simplices of $K$. Note that this implies that each simplex of $K$ is contained in at most $\alpha(n)$ many $n$-simplices of $K$. For convenience, we identify $I^n$ with $|K|$ so that $f$ is a map $f:|K|\to \mcv^m(\scru)$. 

Recall that $K_n$ is the set of all $n$-simplices of $K$, and fix a function $\ell:K_n\to \scru$ which assigns to each $n$-simplex $\sigma$ an open set $U\in \scru$ such that $f(\sigma)\subseteq \wt M_U$. For each simplex $\sigma\in K$, let $\mcs_\sigma$ be the set of $n$-simplices in $K$ which contain $\sigma$. By construction, $\card(\mcs_\sigma)\leq \alpha(n)$ for all $\sigma\in K$ with $\sigma\neq\emptyset$.

We produce a homotopy $H$ between $f$ and a map $g$ by induction on the $k$-skeleton of $K$. So let $k<n$ and suppose there exists a homotopy $H_k$ between $f$ and a map $f_k:|K|\to \mcv^m(\scru)$ such that 
\begin{enumerate}[label=(\alph*)]
\item $H_k|_{\p I^n\times I}$ is continuous  with respect to the topology of $\mgv$,
\item for every $\sigma\in K$, $f_k(|\sigma|)\subseteq \bigcap_{U\in\ell(\mcs_\sigma)}\wt M_U$,
\item $f_k$ satisfies the condition that for each $\sigma\in K^{(k)}$, if $U_\sigma$ denotes the open set $U_\sigma=\bigcap\ell(\mcs_\sigma)$, there exists an open set $V_\sigma\subseteq U_\sigma$ such that $d(V_\sigma, U_\sigma^C)>0$ and $f_k(|\sigma|)\subseteq M_{V_\sigma}$. 
\end{enumerate}
Note that condition (c) states that the $k$-skeleton of $K$ is subordinate to the cover $f_k^{-1}M_\scru$, but in a slightly stronger form which allows \Cref{deform} to be applied to the $(k+1)$-simplices of $K$ in the induction step. We also remark that the sets $U_\sigma$ are non-empty for all $\sigma\neq\emptyset$. This is a consequence of the fact that for $\sigma\neq\emptyset$, $\bigcap_{U\in\ell(\mcs_\sigma)}\wt M_U$ is non-empty and each $\wt M_U$ has $\MCP(p,U)$ for $p>1-1/\alpha(n)$, hence (recalling the beginning of the proof of \Cref{deform}) it follows that $\bigcap_{U\in\ell(\mcs_\sigma)}\wt M_U$ has $\MCP(q, U_\sigma)$ for some $q>0$.

To begin the induction step, let $\sigma\in K_{k+1}$ and let $V_{\p\sigma}$ denote the open set $\bigcup_{\tau\in\p\sigma}V_\tau$ where each $V_\tau$ is the open set guaranteed to exist by (c) for the $k$-simplex $\tau\in \p \sigma$. Then $d(V_{\p\sigma}, U_\sigma^C)>0$. Since $\wt M_\scru$ was constructed with respect to the value $p>1-1/\alpha(n)$ and $\ell(\mcs_\sigma)$ is a subset of $\scru$ with at most $\alpha(n)$ elements, we may apply \Cref{deform} to the map
$$f_k|_{|\sigma|}:(\Delta^{k+1},\p\Delta^{k+1})\to \Big(\bigcap_{U\in \ell(\mcs_\sigma)}\wt M_U, M_{V_{\p\sigma}}\Big).$$
Hence $f_k|_{|\sigma|}$ is homotopic rel.\ $\p \Delta^{k+1}$ by a homotopy $H_\sigma':\Delta^{k+1}\times I\to \bigcap_{U\in\ell(\mcs_\sigma)}\wt M_U$ to a map $f_\sigma'$ with image in $M_{V_\sigma}$ where $V_\sigma\subseteq U_\sigma$ satisfies $d(V_\sigma, U_\sigma)>0$. Moreover, if $\sigma\subseteq \p I^n$, then since $f_k|_{|\sigma|}$ is continuous with respect to the topology of $\mgv$, we may assume $H'_\sigma$ is as well. Note that condition (b) of the induction hypothesis is needed to ensure that $f_k|_{|\sigma|}$ has codomain $\bigcap_{U\in \ell(\mcs_\sigma)}\wt M_U$ so that \Cref{deform} can be applied. Let $f':|K^{(k+1)}|\to \mcv^m(\scru)$ be the map defined by $f'|_{|\sigma|}=f'_\sigma$ for all $\sigma\in K_{k+1}$, and let $H'$ be the homotopy between $f_k|_{|K^{(k+1)}|}$ and $f'$ given by $H'|_{|\sigma|\times I}=H'_\sigma$. By \Cref{hep} there exists a retraction $r:|K|\times I\to |K|\times\{0\}\cup |K^{(k+1)}|\times I$ such that $i\circ r(|\sigma|\times I)\subseteq |\sigma| \times I$ for all $\sigma\in K$. Let $f_k\cup H'$ denote the map $|K|\times\{0\}\cup |K^{(k+1)}|\times I\to \mcv^m(\scru)$ such that $f_k\cup H'|_{|K|\times\{0\}}=f_k$ and $f_k\cup H'|_{|K^{(k+1)}|\times I}=H'$. Then $G=(f_k\cup  H')\circ r:|K|\times I\to \mcv^m(\scru)$ defines a homotopy between $f_k$ and the map $f_{k+1}$ defined by $f_{k+1}=G(\cdot, 1)$. Let $H_{k+1}=H_k\cdot G$ be the concatenation of $H_k$ and $G$. Then $H_{k+1}$ is a homotopy between $f$ and $f_{k+1}$ and we now verify that
\begin{enumerate}[label=(\alph*)]
\item $H_{k+1}|_{\p I^n\times I}$ is continuous  with respect to the topology of $\mgv$,
\item for every $\sigma\in K$, $f_{k+1}(|\sigma|)\subseteq \bigcap_{U\in\ell(\mcs_\sigma)}\wt M_U$,
\item $f_{k+1}$ satisfies the condition that for each $\sigma\in K^{(k+1)}$, if $U_\sigma$ denotes the open set $U_\sigma=\bigcap\ell(\mcs_\sigma)$, there exists an open set $V_\sigma\subseteq U_\sigma$ such that $d(V_\sigma, U_\sigma^C)>0$ and $f_{k+1}(|\sigma|)\subseteq M_{V_\sigma}$. 
\end{enumerate}

First, recall that $f_k|_{|\sigma|}$ is continuous relative to $\mgv$ for all $\sigma\in K$ such that $|\sigma|\subseteq \p I^n$, and $H_\sigma'$ is continuous relative to $\mgv$ for all $k+1$ simplices $\sigma$ such that $|\sigma|\subseteq \p I^n$. Then $f_k\cup H'$ restricted to the set
$$A= \p I^n\times I \cap \Big(|K|\times\{0\}\cup|K^{(k+1)}|\times I\Big)$$
is continuous relative to $\mgv$. Since $r(\p I^n\times I)\subseteq A$ holds due to the fact that $i\circ r(|\sigma|\times I)\subseteq |\sigma|\times I$ for all $\sigma\in K$, we see that $G|_{\p I^n\times I}$ is continuous relative to $\mgv$. This implies (a). 

To see that (b) holds, let $\sigma\in K$. If $\sigma\in K^{(k)}$, then $f_{k+1}|_{|\sigma|}=f_k|_{|\sigma|}$ and we are done by the induction hypothesis. If $\sigma\in K^{(k+1)}$, then
$$f_{k+1}(|\sigma|)\subseteq H_{\sigma}'(|\sigma|\times I)\subseteq \bigcap_{U\in\ell(\mcs_\sigma)}\wt M_U.$$
Lastly, suppose $\sigma\in K\setminus K^{(k+1)}$. Then $f_{k+1}(|\sigma|)\subseteq G(|\sigma|\times I)$, which is in turn contained in
$$f_k(|\sigma|)\cup H'\big(|\sigma^{(k+1)}|\times I\big)$$
by the fact that $G=(f\cup H')\circ r$ and $i\circ r(|\sigma|\times I)\subseteq |\sigma|\times I$ for all $\sigma\in K$. Observe that $f_k(|\sigma|)\subseteq \bigcap_{U\in\ell(\mcs_\sigma)}\wt M_U$ and 
$$H'\big(|\sigma^{(k+1)}|\times I\big)\subseteq\bigcup_{\tau\in\sigma_{k+1}}\bigcap_{U\in\ell(\mcs_\tau)}\wt M_U\subseteq \bigcap_{U\in\ell(\mcs_\sigma)}\wt M_U$$
where $\tau\in\sigma_{k+1}$ denotes that $\tau$ is a $(k+1)$-simplex of $\sigma$, and the last subset inclusion follows from the fact that if $\tau\subseteq \sigma$, then $\ell(\mcs_\sigma)\subseteq\ell(\mcs_\tau)$, and so $\bigcap_{U\in\ell(\mcs_\tau)}\wt M_U\subseteq \bigcap_{U\in\ell(\mcs_\sigma)}\wt M_U$. Hence (b) follows.

Lastly, (c) follows directly by the construction of $f_{k+1}$. Thus we have completed the induction step.

The base case is satisfied trivially by starting induction at $k=-1$, setting $f_{-1}=f$, and letting $H_{-1}$ be the trivial homotopy $H_{-1}(x,t)=f(x)$. Here we use the convention that $K^{(-1)}=\{\emptyset\}$ and $d(\emptyset, A)=\infty$ for any subset $A\subseteq X$ so that condition (c) is satisfied by setting $V_\sigma=\emptyset$ for $\sigma=\emptyset\in K^{(-1)}$.

Hence by induction there exists a homotopy $H=H_n$ between $f$ and a map $g=f_n$ which satisfies conditions (1) and (2) in the statement of the lemma.
\end{proof}

\begin{lemma}\label{second}
Let $n\in\bbn$ and let $f:I^n\to\mcv^m(\scru)$ be a map such that $f|_{\p I^n}$ is continuous with respect to the topology of $\mgv$. If $I^n$ admits a triangulation subordinate to $f^{-1}M_\scru$, then there exists a homotopy $H:I^n\times I\to \vm$ between $f$ and a map $g:I^n\to \mgv$ such that $H|_{\p I^n\times I}$ is continuous with respect to the topology of $\mgv$.
\end{lemma}

\begin{proof}
Suppose $T:|K|\to I^n$ is a triangulation subordinate to $f^{-1}M_\scru$. To simplify notation, for any simplex $\sigma\in K$, we identify $|\sigma|$ with its image in $I^n$ under $T$. Given an $n$-simplex $\sigma\in K$, write $\sigma=[v_0, \dots, v_n]$. Express each point of $x\in |\sigma|$ in barycentric coordinates by writing $x=\sum_{i=0}^nx_iv_i$. Let $g_\sigma:\Delta^n\to \mcv^m(\scru)$ be the map defined by $g_\sigma(\sum_{i=0}^nx_iv_i)=\sum x_if(v_i)$. For each $n$-simplex $\sigma\in K$, $g_\sigma$ is continuous with respect to the topology of $\mgv$ because $g_\sigma$ is an affine mapping of $\Delta^n$ into the simplex of $\mgv$ with vertex set $\bigcup_i\supp f(v_i)$. Note that $\bigcup_i\supp f(v_i)$ does in fact define a simplex of $\vu$ due to the fact that $T$ is subordinate to $f^{-1}M_\scru$, which implies that $f(|\sigma|)\subseteq M_U$ for some $U\in\scru$, and so $\bigcup_i\supp f(v_i)\subseteq U$. The continuity of $\id:\mgv\to\vm$ implies that $g_\sigma$ is also continuous with respect to the topology of $\vm$. 

For all $\sigma\in K_n$, let $H_\sigma:\Delta^n\times I\to \mcv^m(\scru)$ be the linear homotopy between $f|_{|\sigma|}$ and $g_\sigma$. Since $T$ is subordinate to $f^{-1}M_\scru$, for each $\sigma\in K$, there exists $U_\sigma\in\scru$ such that $f(x)\in M_{U_\sigma}$ for all $x\in |\sigma|$. Then for any $t\in I$, we have that 
$$H_\sigma(x,t)=(1-t)f(x)+tg(x)=(1-t)f(x)+t\sum_{i=0}^n x_if(v_i)\in M_{U_\sigma}.$$
Hence $H_\sigma$ is well-defined, and is thus continuous by \Cref{linear}. It is straightforward to see that if $\sigma\cap\sigma'=\tau$, then $H_\sigma|_{|\tau|}=H_{\sigma'}|_{|\tau|}$. So let $H:I^n\times I\to \mcv^m(\scru)$ be the map defined by $H|_{|\sigma|}=H_\sigma$ for all $\sigma\in K_n$. Then $H$ is a homotopy between $f$ and the map $g:I^n\to \mcv^m(\scru)$ satisfying $g|_{|\sigma|}=g_\sigma$. Moreover, if $\tau\in K$ satisfies $|\tau|\subseteq \p I^n$, then $f|_{|\tau|}$ is continuous with respect to the topology of $\mgv$. Since $g|_{|\tau|}$ is continuous with respect to the topology of $\mgv$, so too is $H|_{|\tau|\times I}$ by \Cref{CW-linear}. Hence $H|_{\p I^n\times I}$  is continuous with respect to the topology of $\mgv$.
\end{proof}

\begin{lemma}\label{lifting}
Let $f:I^n\to \vm$ be a map such that $f|_{\p I^n}$ is continuous with respect to the topology of $\mgv$. Then $f$ is homotopic rel.\ $\p I^n$ to a map $g:I^n\to \mgv$.
\end{lemma}

\begin{proof}
Suppose $f:I^n\to \vm$ is a map such that $f|_{\p I^n}$ is continuous with respect to the topology of $\mgv$. We may then apply \Cref{injpart1} to find a homotopy $H:I^{n}\times I\to \mcv^m(\scru)$ between $f$ and a map $f':I^n\to \vm$ such that $H|_{\p I^{n}\times I}$ is continuous with respect to the topology of $\mgv$ and $I^{n}$ admits a triangulation subordinate to $f'^{-1}M_\scru$.

We have that $f'|_{\p I^{n+1}}$ is continuous with respect to the topology of $\mgv$ since $H|_{\p I^{n}\times I}$ is. Hence we may apply \Cref{second} to obtain a homotopy $H':I^{n}\times I\to \mcv^m(\scru)$ between $f'$ and a map $f'':I^{n}\to \mgv$ such that $H'|_{\p I^n\times I}$ is continuous relative to $\mgv$. Let $G:I^{n}\times I\to\mcv^m(\scru)$ be the concatenation of $G=H\cdot H'$. Let $A=\p I^{n}\times I\cup I^{n}\times \{1\}$ and let $\phi:A\to I^{n}\times \{0\}$ be a homeomorphism which fixes $\p I^{n}\times\{0\}$. Finally, let $g=G|_A\circ \phi^{-1}$. Then $G$ induces a homotopy rel.\ $\p I^n$ between $f$ and $g$. Because $G|_{\p I^n\times I}$ and $f''$ are both continuous with respect to the topology of $\mgv$, so too is $g$.
\end{proof}

We now prove our main result.

\begin{recallthm}
For any uniformly bounded open cover $\scru$ of a metric space $X$, the natural bijection $\id:\gv\to \mcv^m(\scru)$ is a weak homotopy equivalence.
\end{recallthm}

\begin{proof}
Combining \Cref{weak-lemma} and \Cref{lifting}, we see that $\id:\mgv\to\mcv^m(\scru)$ is a weak homotopy equivalence. Since $\id:|K|\to |K|_m$ is a homotopy equivalence for any simplicial complex $K$ \cite[Appendix 1, Theorem 10]{MS82}, in particular for the case $K=\vu$, the theorem follows.
\end{proof}

To strengthen the weak homotopy equivalence of \Cref{mainthm} to a homotopy equivalence, it suffices to show that $\vm$ has the homotopy type of a CW complex. We reiterate a question posed in \cite{Adams}.

\begin{question}
If $\scru$ is a uniformly bounded open cover of a metric space $X$, does $\vm$ have the homotopy type of a CW complex? More specifically, is $\vm$ an absolute neighborhood retract?
\end{question}

\section{The persistent homology of Vietoris metric thickenings}\label{persistence}

A \textit{persistence module} is a family of $\bbk$-vector spaces $\{V_r\}_{r\in\bbr}$, where $\bbk$ is any field, along with a collection of linear maps $\phi_{r,s}:V_r\to V_s$ for every choice of $r\leq s$ such that $\phi_{r,r}=\id$ and $\phi_{r,t}=\phi_{s,t}\circ\phi_{r,s}$ whenever $r\leq s\leq t$. More succinctly, a persistence module is a functor $(\bbr, \leq)\to \Vect_\bbk$ from the poset $(\bbr, \leq)$ to the category of $\bbk$-vector spaces. 

For a fixed metric space $X$, note that $|{\VR(X;\bullet)}|$ and $\VR^m(X;\bullet)$ define functors from $(\bbr, \leq)$ to $\Top$, the category of topological spaces. Let $H_n:\Top\to\Vect_\bbk$ denote the $n$-th singular homology functor with coefficients in $\bbk$. Then for any metric space $X$ and $n\in\bbn$, $H_n\circ |{\VR(X;\bullet)}|$ and $H_n\circ \VR^m(X;\bullet)$ are both examples of persistence modules.

Note that the following diagram, in which the horizontal arrows are the inclusions, commutes for any $r_1\leq r_2$.
$$\begin{tikzcd}
\vert{\VR(X;r_1)}\vert \arrow[d, "\id"'] \arrow[r, hook] & \vert{\VR(X;r_2)}\vert \arrow[d, "\id"] \\
\VR^m(X;r_1) \arrow[r, hook]                           & \VR^m(X;r_2)                        
\end{tikzcd}$$
Additionally, recall that a weak homotopy equivalence $X\to Y$ induces an isomorphism on homology groups for any choice of coefficients \cite[Proposition 4.21]{Hatcher}. Hence \Cref{mainthm} and the above observations imply that the functors $H_n\circ |{\VR(X;\bullet)}|$ and $H_n\circ \VR^m(X;\bullet)$ are naturally isomorphic. The same is true if we replace $\VR(X;\bullet)$ and $\VR^m(X;\bullet)$ with $\C(X;\bullet)$ and ${\C}{}^m(X;\bullet)$ respectively. Hence we have the following.

\stepcounter{recallthm}
\begin{recallcor}
For any metric space $X$, the persistence modules $H_n\circ |{\VR(X;\bullet)}|$ and $H_n\circ \VR^m(X;\bullet)$ (respectively, $H_n\circ |{\C(X;\bullet)}|$ and $H_n\circ {\C}{}^m(X;\bullet)$) are isomorphic.
\end{recallcor}

\section*{Acknowledgments}

The author would like to thank Vasileios Maroulas and Henry Adams for helpful discussion.

\end{document}